\documentclass[11pt,letterpaper]{amsart}
\usepackage[usenames,dvipsnames]{color}
\usepackage{amsthm,amsfonts,amssymb,amsmath,amsxtra}
\usepackage[all]{xy}
\SelectTips{cm}{}
\usepackage{xr-hyper}
\usepackage[colorlinks=
   citecolor=Black,
   linkcolor=Red,
   urlcolor=Blue]{hyperref}
\usepackage{verbatim}
\usepackage{caption}
\usepackage{tikz-cd}
\usepackage{lineno}
\usepackage{mathrsfs}

\usepackage{mathrsfs}

\RequirePackage{xspace}
\RequirePackage{etoolbox}
\RequirePackage{varwidth}
\RequirePackage{enumitem}
\RequirePackage{tensor}
\RequirePackage{mathtools}
\RequirePackage{longtable}
\RequirePackage{multirow}

\newtheorem{theorem}{Theorem}
\newtheorem{proposition}[theorem]{Proposition}
\newtheorem{lemma}[theorem]{Lemma}

\newtheorem{corollary}[theorem]{Corollary}

\theoremstyle{definition}
\newtheorem{definition}[theorem]{Definition}

\newtheorem{remark}[theorem]{Remark}

\numberwithin{equation}{section}
\numberwithin{theorem}{section}

\newcommand{\BF}{\ensuremath{\mathbb {F}}\xspace}
\newcommand{{\BG}}{\ensuremath{\mathbb {G}}\xspace}

\newcommand{{\BK}}{\ensuremath{\mathbb {K}}\xspace}

\newcommand{\BQ}{\ensuremath{\mathbb {Q}}\xspace}
\newcommand{\BR}{\ensuremath{\mathbb {R}}\xspace}

\newcommand{\BW}{\ensuremath{\mathbb {W}}\xspace}

\newcommand{\CA}{\ensuremath{\mathcal {A}}\xspace}
\newcommand{\CB}{\ensuremath{\mathcal {B}}\xspace}

\newcommand{\CG}{\ensuremath{\mathcal {G}}\xspace}

\newcommand{\CN}{\ensuremath{\mathcal {N}}\xspace}
\newcommand{\CO}{\ensuremath{\mathcal {O}}\xspace}

\newcommand{\CR}{\ensuremath{\mathcal {R}}\xspace}

\def\sc{{\rm sc}}

\def\der{{\rm der}}
\def\sc{{\rm sc}}
\def\brk{{\breve k}}
\def\COk{{\CO_{\brk}}}

\def\ov{\overline}
\def\i{^{-1}}
\title{Mackey}
\author{Zhihang Yu}
\address{The Second Affiliated Hospital of Chongqing Medical Univesity, Chongqing 400010, China}
\address{The National Center for Applied Mathematics in Chongqing, Chongqing 401331, China}
\email{yuzhihang@amss.ac.cn}

\begin{document}
\title{Scalar Products of Generic Parahoric Lusztig Inductions}
\maketitle
\begin{abstract}
    Parahoric Lusztig induction gives a broad class of virtual smooth representations of parahoric subgroups in a $p$-adic group,
    serving as a natural generalization of classical Lusztig induction. This construction has useful applications in the representation theory of $p$-adic groups. In this paper, we prove a scalar product formula for parahoric Lusztig induction under the $(L,G)$-generic condition, which can be viewed as a parahoric analogue of a classical result of Lusztig in 1976. As an application, we describe the irreducible decomposition of the parahoric Deligne–Lusztig representation in the case of elliptic tori.
\end{abstract}
\section{Introduction}
\subsection{Background and Motivation}
Let $G$ be a reductive group over $\overline{\mathbb{F}}_q$ with an $\mathbb{F}_q$-rational structure, and let $F$ be the corresponding geometric Frobenius endomorphism. We write $G^F$ for the finite group of Lie type consisting of the $F$-fixed points. The Deligne-Lusztig construction \cite{DL} produces virtual representations $R_{T,B}^G(\theta)$ of $G^F$ via $\ell$-adic cohomology, where $T$ is an $\mathbb{F}_q$-rational maximal torus and $\theta\in \widehat{T^{F}}$. These (virtual) representations, known as Deligne-Lusztig representations, form the cornerstone of the representation theory of finite reductive groups --- the celebrated Deligne-Lusztig theory.

The Deligne-Lusztig construction naturally extends to the $p$-adic setting, which yields parahoric Deligne-Lusztig varieties (deep level Deligne-Lusztig varieties). Their cohomology similarly gives rise to virtual representations $R_{T,B,r}^{G}(\theta)$ for parahoric subgroups, called parahoric Deligne-Lusztig representations (deep level Deligne-Lusztig representations). Lusztig \cite{Lu79} initiated their study in 1979. In the subsequent decades, it has been a long-standing program to study these deep level generalisations. The interested reader may refer to the works \cite{BoyarchenkoW_16, Chan_siDL, CI_loopGLn, ChenS_17, ChenS_23, CO, ChanOi_25b, Nie, IvanovNie_24, IvanovNie_25} among many others in this flourishing field.

Let $G$ be a connected reductive algebraic group over $\overline{\mathbb{F}}_q$ equipped with an $\mathbb{F}_q$-rational structure, and let $F: G \to G$ denote the corresponding geometric Frobenius endomorphism.  We recall the Lusztig induction functors and the resulting (virtual) representation $R_{L,P}^G(\theta)$\cite{Lu76}, where $L$ is an $F$-stable Levi subgroup, $P$ a parabolic subgroup with Levi factor $L$, and $\theta$ a (virtual) representation of $L^F$. 

A fundamental problem is to compute the inner products $$\big\langle R^G_{L,P}(\theta), R^G_{L',P'}(\theta')\big\rangle_{G^F}$$ A standard approach to such questions is through the Mackey formula, the conjectural Mackey formula takes the form
\begin{equation}\label{Mackey}
^{\ast}R^G_{L,P}\circ R^G_{L',P'} (\theta)= \sum_{w\in L^F\backslash S(L,L')^F/L'^F} R^L_{L\cap\,^wL'}\circ\,^\ast R^{^wL'}_{L\cap\,^wL'}(^w\theta).
\end{equation}
where $S(L,L')=\{x\in G\mid L\cap\,^xL'\text{ contains a maximal torus}\}$ and the functor $^{\ast}R^G_{L,P}$ is defined as both the left and right adjoint to $R^G_{L,P}$. The strongest current results toward (\ref{Mackey}) come from Bonnafé-Michel and Taylor \cite{Comp,Taylor}. On the other hand, scalar product formulas related to (\ref{Mackey}) are known in several important cases. In particular, Lusztig \cite{Lu76} proved that when $L, L'$ and $\theta, \theta'$ satisfy certain conditions, the following scalar product formula holds:
\begin{equation}\label{Lu76}
\big\langle R^G_{L,P}(\theta), R^G_{L,P}(\theta')\big\rangle_{G^F} = \big\langle \theta, \theta'\big\rangle_{L^F}
\end{equation}

In the deep level (parahoric) setting, the naive analogue of \eqref{Mackey} generally fails to hold. Nevertheless, the Mackey formula \eqref{Mackey} continues to hold when \(\theta\) is generic:
this was shown by Lusztig \cite{Lu04}, by Stasinski \cite{Sta09}, and by Chan--Ivanov \cite{CI21} when \(L\) and \(L'\) are tori,
and later extended by Chan \cite{Ch24} to cases where one of them is a Levi subgroup.
The case where both \(L\) and \(L'\) are arbitrary Levi subgroups remains considerably more difficult,
since some arguments valid for the torus case no longer apply in this more general setting.

In this paper, we obtain an analogue of Lusztig's equality \eqref{Lu76} under the $(L,G)$-generic condition (see Definition~\ref{generic2}), which is quite different from Lusztig's original condition. 
The key idea is to adapt Lusztig’s original methods to the deep level setting and combine them with a careful analysis of the relevant cohomology, which allows us to treat the case of Levi subgroups.

\subsection{Main Result}
Let $k$ be a non-archimedean local field with residue field $\mathbb{F}_q$ of cardinality $q$ and  characteristic $p$. Let $\brk$ denote the completion of a maximal unramified extension of $k$, and let $F$ be the Frobenius automorphism of $\brk/k$.

Let $G$ be a $k$-rational reductive group, and let $\mathbf{x}$ be a point in its Bruhat-Tits building over $k$. Let $T \subseteq G$ be a $k$-rational maximal torus and splitting over $\brk$ such that $\mathbf{x}$ lies in its apartment over $\brk$. Let $V$ be the unipotent radical of a $\brk$-rational parabolic subgroup with Levi factor $L$ containing $T$. 

For any integer $r \geq 0$, we associate to $L$ a subgroup $L_r\subset G_r$ and an $\overline{\mathbb{F}}_q$-variety $X^G_{L,P,r}$
called a \emph{parahoric Lusztig variety}. The group $G_r^F$ (resp. $L_r^F$) acts on $X^G_{L,P, r}$ via left (resp. right) multiplication, inducing an action on the $\ell$-adic cohomology  $H_c^i(X^G_{L,P,r}, \overline{\mathbb{Q}}_\ell)$, where $\ell \neq p$ is a fixed prime. 

Let $\psi$ be a (virtual) representation of $L^F_r$. We define the  \emph{parahoric Lusztig induction} as follows
\[ R_{L,P,r}^G(\psi)(g) = \frac{1}{|L_r^F|}\sum_{l\in  L_r^F}  Tr((g,l)|H_c^\ast(X^G_{L,P,r}, \overline{\mathbb{Q}}_\ell))\psi(l^{-1}) \]

This construction already appears in the literature (see, e.g.,~\cite{Crisp}). In particular, it was systematically studied by Chan~\cite{Ch24}.  It is a natural generalization of the classical Lusztig induction~\cite{Lu76}. When \(r=0\), this construction recovers the classical Lusztig induction.

We now state the main theorem.
\begin{theorem}\label{introthm}(Theorem~\ref{thm})
Suppose that $p$ satisfies condition $(\ast)$, i.e.,
$p\neq 2, p \nmid |\pi_1(G_\der)|$ and $p$ is not a bad prime for $G$
. Then
$$\big\langle R^{G}_{L, P,r}(\psi'), R^{G}_{L, P,r}(\psi)\big\rangle_{G^F_r} = \big\langle \psi', \psi\big\rangle_{L^F_r}$$
where $\psi, \psi'$ are two $(L,G)$-generic \footnote{It is a nontrivial condition on the restriction of $\psi$ to $\mathrm{ker}(T^F_r\to T^F_{r-1})$. See Definition~\ref{generic}} (virtual) representations of $L^F_r$, and both are constituents of $R^{L}_{T, B\cap L,r}(\theta)$ and $\theta$ is an $(L,G)$-generic character of $T^F_r$.
\end{theorem}

\begin{theorem}\label{introEllthm}(Theorem~\ref{Ellthm})
Let $p$ satisfy condition $(\ast)$. Suppose $T$ is elliptic. Then:
$$\big\langle R^{G}_{L, P,r}(\psi'), R^{G}_{L, P,r}(\psi)\big\rangle_{G^F_r} = \big\langle \psi', \psi\big\rangle_{L^F_r}$$
where $\psi, \psi'$ are two constituents of $R^{L}_{T, B\cap L,r}(\theta)$ and $\theta$ is an $(L,G)$-generic character of $T^F_r$
\end{theorem}

\begin{remark}
In Theorem~\ref{introEllthm}, the hypotheses can be weakened to require that $\psi$ (resp.\ $\psi'$)  appears as a constituent in $R^{G}_{T,B\cap L,r}(\theta)$ (resp.\ $R^{G}_{T',B'\cap L,r}(\theta')$). Here $\theta$ is $(L,G)$-generic, and $(T,\theta)$ is geometrically conjugate to $(T',\theta')$ in the sense of \cite{DL}. Our argument remains valid in this setting. 
\end{remark}

\subsection{Application: Irreducible Decomposition of Parahoric Deligne-Lusztig}
We invoke a result of Howe \cite{Howe} and Kaletha \cite{Kal} on Howe factorizations of smooth characters $\theta: T^F \to \ov \BQ_\ell^\times$ of depth $r$, which establishes that if $p$ is not a bad prime for $G$ then there is a datum \[(G^i, \phi_i, r_i)_{0 \le i \le d},\] where $T \subseteq G^0 \subsetneq \cdots \subsetneq G^d = G$ are $k$-rational Levi subgroups of $G$ and each $\phi_i: (G^i)^F \to \ov \BQ_\ell^\times$ is an $(L,G)$-generic character of depth $r_i$, such that \[\phi_{-1} := \theta \prod_{i=0}^d \phi^{-1}_i |_{T^F}\] is a character (of $T^F$) of depth $0$.
 
\begin{theorem}\label{Decomp}(Theorem~\ref{decomp})
 Let $p$ satisfy condition $(\ast)$, and suppose $T$ is elliptic. Then 
    $$R^G_{T,B,r}(\theta)=R^G_{G^0,P^0,r}(R^{G^0}_{T,B\cap G^0,0}(\phi_{-1})\otimes\theta^{\geq 0})=\sum_{\rho} m_{\rho}R^G_{G^0,P^0,r}(\rho\otimes\theta^{\geq 0})$$
    where $\theta^{\geq 0}=\prod_{i=0}^d\phi^{}_i |_{(G^0_r)^F}$ and $R_{T, B\cap G^0, 0}^{G^0}(\phi_{-1})$ is a classical Deligne-Lusztig representation of the reductive quotient of $(G^0_r)^F$, and $\rho$ ranges over its irreducible summands with multiplicity $m_\rho$.
    
    Moreover, the summands $R^G_{G^0,P^0,r}(\rho\otimes\theta^{\geq 0})$ are pairwise non-isomorphic irreducible representations of $G^F_r$ up to sign.
\end{theorem}
\begin{remark}
    
    In \cite[Theorem 7.5]{Nie} Nie also obtains the irreducible decomposition for elliptic parahoric Deligne-Lusztig representations. His approach is quite different from ours, which is based on a variety coming from G{\'e}rardin's representation \cite{Gerardin}. Our work provides a new proof within the original framework of parahoric Deligne-Lusztig representations. 
    
    Furthermore, our approaches also show that Theorem~\ref{Decomp} remains valid even when $(T,\theta)$ is "split-generic" in the sense of \cite[Definition 6.1]{Ch24}. 
\end{remark}
\subsection{Structure of the paper}
The paper is organized as follows. In \S\ref{2}, we present some essential properties of parahoric Lusztig induction. In \S\ref{3}, we introduce the key variety $\Sigma^M$, decompose it into strata, and analyze the cohomology of each stratum. Combining these results, we establish our main result Theorem~\ref{thm}. In \S\ref{4}, we decompose elliptic parahoric Deligne-Lusztig representation $R^G_{T,B,r}(\theta)$ into irreducible representations and finish the proof of Theorem~\ref{Decomp}.

\subsection*{Acknowledgement}
We would like to thank Sian Nie for helpful discussions and encouragement. We are also grateful to Pengcheng Li for fruitful discussions. We also wish to acknowledge our reliance on several key results in \cite{Ch24}, which play an essential role in our computation.

\subsection*{Conventions and notation}
Recall that $k$ is a non-archimedean field with finite residue field $\BF_q$ of cardinality $q$ and  characteristic $p$. And $\brk$ is the completion of a maximal unramified extension of $k$. Denote by $\CO_k$ and $\COk$ the integer rings of $k$ and $\brk$ respectively. Fix a uniformizer $\varpi \in \CO_k$, and denote by $\mathrm{val}:\brk \to \mathbb{Z}\cup\{\infty\}$ the corresponding valuation. Let $F$ be the Frobenius automorphism of $\brk$ over $k$.

Let $G$ be a connected reductive group splitting over $\brk$. We denote by $\mathrm{Lie}(G)$ the Lie algebra and $\mathrm{Lie}^*(G)$ the dual of $\mathrm{Lie}(G)$. We write $Z(G)$ for the center of $G$, $G_\der$ for the derived subgroup of $G$, and $G_\sc$ for the simply connected covering of $G_\der$.
Let $\mathscr{B}(G,k)$ denote the (enlarged) Bruhat-Tits building of a reductive group $G$ over $k$. Bruhat-Tits theory associates to each point $\mathbf x \in \mathscr{B}(G,k)$:

\begin{itemize}
    \item A connected parahoric $\mathcal{O}_k$ model $\mathcal{G}_{\mathbf x}$ of $G$
    \item The Moy-Prasad filtration subgroups $\mathcal{G}_{\mathbf x}^r$ for $r \in \widetilde{\mathbb{R}}_{\geq 0}$, where 
    \[
    \widetilde{\mathbb{R}} = \mathbb{R} \sqcup \{r+ \mid r \in \mathbb{R}\}
    \]
    is equipped with the ordering $s < s+ < r$ for any real numbers $s < r$
\end{itemize}
We have associated filtrations $\mathrm{Lie}(G)^r_\mathbf{x}$ and $\mathrm{Lie}^*(G)^r_\mathbf{x}$ which are stable under the adjoint and coadjoint actions of $\mathcal{G}_\mathbf{x}^0$, respectively.

For $s \le r \in \widetilde\BR_{\ge 0}$ we denote by $G_r^s$ the $\BF_q$-rational smooth affine group scheme, which represents the perfection of the functor \[R \mapsto \CG_\mathbf{x}^s(\BW(R)) / \CG_\mathbf{x}^{r+}(\BW(R)),\] where $R$ is an $\BF_q$-algebras, and $\BW(R)$ is the Witt ring of $R$ if ${\rm char}~k = 0$ and $\BW(R) = R[[\varpi]]$ otherwise. 

Let $H \subseteq G$ be a closed $\brk$-rational subgroup. We denote by $H_r^s \subseteq G_r^s$ the closed subgroup defined in \cite[\S 2.6]{CI21}. If $H$ is $k$-rational, the (geometric) Frobenius $F$ acts naturally on $H_r$. By abuse of notation, we will write $H = H(\brk)$, $G_{\mathbf x}^s = \CG_{\mathbf x}^s(\COk) \subseteq G$ and $H_r = H_r(\ov\BF_q)$. In particular, $G^F = G(k)$ and $G_r^F = G_r(\BF_q) = G_{\mathbf x}^F / (G_\mathbf{x}^{r+})^F$.

Let $T \subseteq G$ be a $k$-rational maximal torus splitting over $\brk$. We denote by $\CA(T, \brk)$ the apartment of $T$ inside the Bruhat-Tits building $\CB(G, \brk)$ of $G$ over $\brk$. Write $\Phi(G, T)$ for the root system of $T$ in $G$ over $\brk$. Throughout our paper, we fix a point ${\mathbf x} \in \CA(T, \brk)$. 

For convenience, we collect here the assumptions on $p$ placed in different section of the paper. In $\S2$, there are no further assumptions. In $\S3-\S4$ we assume $p$ satisfies the following condition:  \[\tag{$\ast$} \text{$p\neq 2$, $p \nmid |\pi_1(G_\der)|$ and $p$ is not a bad prime}\] 
This condition will be referred to as condition ($\ast$).

Note that a prime is not bad for $G$ if it is not bad for all irreducible components of its absolute root system. Here we list not bad primes for irreducible root system:
\[
\begin{array}{c|c|c}
\text{\(B_n,\, C_n,\, D_n\)} & \text{\(E_6,\, E_7,\, F_4,\, G_2\)} & \text{\(E_8\)} \\
\hline
p\neq 2 & p\neq2,\,3 & p\neq2,\,3,\,5
\end{array}
\]

All (virtual) representations of groups in this paper are taken with coefficients in $\ov\BQ_\ell$, where $\ell \neq p$ is a prime number. A one-dimensional representation is called a \emph{character}. Let $H$ be an abstract group. For two subgroups $H_1, H_2 \subseteq H$, let $[H_1, H_2]$ denote the subgroup generated by the commutators $[H_1, H_2] := h_1 h_2 h_1^{-1} h_2^{-1}$ for all $h_1 \in H_1$ and $h_2 \in H_2$. For $h,k \in H$, we write ${}^h k = h k h^{-1}$. If $H$ is finite, we write $\CR(H)$ for the Grothendieck group of $\overline{\mathbb{Q}}_\ell[H]$-module or equivalently the set of virtual representations of $H$. For \(\psi\in \CR(H)\) and \(h\in H\), we write \(\psi(h)\) for the value of the class function associated to \(\psi\) at \(h\).
Let $H'$ be a quotient group of $H$, the inflation functor is denoted by $\mathrm{Inf}^{H}_{H'}$. For a right \(\overline{\mathbb{Q}}_\ell[H]\)-module \(V\) and a left \(\overline{\mathbb{Q}}_\ell [H]\)-module \(W\), we simply write 
\(V \otimes_H W\) for the tensor product 
$
 V \otimes_{\overline{\mathbb{Q}}_\ell[H]} W$ which can be extended by linearity to virtual representations.

\section{Parahoric Lusztig Induction}\label{2}
\subsection{Definitions}
\begin{definition}
Let $L \subseteq G$ be a $k$-rational Levi subgroup containing $T$. Let $P = L V \subseteq G$ be a parabolic subgroup with $V$ the unipotent radical and $L$ the Levi factor. Let $r \in \BR_{\ge 0}$. The associated parahoric Lusztig variety is defined by \[ X^G_{L, P, r}  = \{x \in G_r; x\i F(x) \in F (V_r)\}.\]
Note that there is a natural action of $G_r^F\times L_r^F$ on $X_{L, P, r}$ given by $$(g, l): x \mapsto g x l.$$
\end{definition}

\begin{definition}
We define the parahoric Lusztig induction functor
$$R_{L, P, r}^G:\CR(L^F_r)\to \CR(G^F_r)$$
by
\begin{align*}
R_{L, P, r}^G(\psi)(g)&:=(H_c^*(X^G_{L, P, r}, \ov \BQ_\ell)\otimes_{L^F_r}\psi)(g)\\
 &=\frac{1}{|L_r^F|}\sum_{l\in  L_r^F}  Tr((g,l)|H_c^\ast(X^G_{L,P,r}, \overline{\mathbb{Q}}_\ell))\psi(l^{-1})
,\quad \psi\in \CR(L^F_r).
\end{align*}
\end{definition}

\subsection{Properties}
We show some known properties of parahoric Lusztig induction functors.

\begin{proposition}\label{trans}\cite[\S 3.1]{Ch24}
Let $ Q\subset  P$ be two parabolic subgroups of $ G$, and let $ M\subset  L$ be $k$-rational Levi subgroup of $ Q,  P$. We have an isomorphism
    $$X^{G}_{L, P, r}\times_{L^{F}_r}X^{L}_{M, Q\cap L,r}\cong X^{G}_{M, Q,r}$$
    $$(v, u)\mapsto vu.$$
\end{proposition}

\begin{corollary}(Transitivity)
    Let $ Q\subset  P$ be two parabolic subgroups of $ G$, Let $ M\subset  L$ be $k$-rational Levi subgroup of $ Q,  P$. Then
    $$R^{G}_{L, V,r}\circ R^{L}_{M, L\cap Q,r}=R^{G}_{M, Q,r}.$$
\end{corollary}
 
\begin{lemma}
   Let $n$ be an integer such that $F^n(V_r)=V_r$. $X^G_{L, P, r}$ is separated, and (the perfection of a) smooth scheme of finite type over $\BF_{q^n}$.
\end{lemma}
\begin{proof}
    We follow the argument in \cite[\S 3.1]{CI21}. $X^G_{L, P, r}$ is the pullback under Lang's map $G_r\to G_r$, $g\mapsto g\i F(g)$ which is a finite \'etale map, and $V_r$ is isomorphic to (the perfection of) an affine space.
\end{proof}

Throughout the paper, for a (virtual) representation $\psi$, $\overline{\psi}$ denotes its contragredient.

\begin{proposition}\label{bar}
  Let $\psi \in \mathcal{R}(L_r^F)$. Then $$ R_{L, P, r}^{G}(\overline \psi)=\overline{
  R_{L, P, r}^{G}(\psi) }$$
\end{proposition}
\begin{proof}
    For any $g\in G^F_r$
   \begin{align*}
       R_{L, P, r}^{G}(\overline \psi)(g)&=|L_r^{F}|^{-1}\sum_{l\in L_r^{F}}{Tr((g,l)|H^{*}_c(X^{G}_{L, P,r})}\overline\psi(l^{-1})\\
    &=|L_r^{F}|^{-1}\sum_{l\in L_r^{F}}\overline{Tr((g,l)|H^{*}_c(X^{G}_{L, P,r})}\overline\psi(l^{-1})\\
    &=\overline{R^G_{L,P,r}(\psi)(g)}
    \end{align*} 
\end{proof}

Let $\pi: G_\sc \to G_\der$ be the covering map. For $T\subset L\subset P=LV\subset G$, consider the following groups: $ T_\der:= G_\der \cap T$, $ L_\der:= G_\der\cap L$, $ P_\der:= G_\der \cap  P$, and $ V_{\der}= G_{\der}\cap V= V$. We also set $ T_{\sc}=\pi^{-1}( T_{\der})$, $ L_{\sc}=\pi^{-1}( L_\der)$, $ P_\sc=\pi^{-1}( P_{\der})$, and $\pi_r:G_{\sc,r}\to G_{\der,r}$ to be the map induced by $\pi$.

\begin{lemma}\label{pisc}
    There is a natural bijection $$ T^F_r/\pi_r((T_{\sc,r})^F)\cong L^F_r/\pi_r((L_{\sc,r})^F)\cong G^F_r/\pi_r((G_{\sc,r})^F).$$
\end{lemma}
\begin{proof}
    It suffices to prove $ T^F_r/\pi_r((T_{\sc,r})^F)\cong G^F_r/\pi_r((G_{\sc,r})^F)$.
    Injectivity is clear; we need to check surjectivity of the map $(T_r\times G_{\sc,r})^F\to G^F_r$ induced by
    $$\varphi:T_r\times G_{\sc,r}\to G_r$$
    $$(t, g)\mapsto t\pi_r(g)$$
    via $\varphi$, $T_r\times G_{\sc,r}$ is a $T_{\sc,r}$-torsor over $G_r$, where $T_{\sc,r}$ acts on $\varphi:T_r\times G_{\sc,r}$ via $t'\cdot (t, g)=(t\pi_r(t'), t'^{-1}g)$. And one applies Lang's theorem to the fiber of $\varphi$ which is a connected group $T_{\sc,r}$, we obtain the surjectivity.
\end{proof}

   \begin{remark}
   The image $\pi_r((G_{\sc,r})^F)$ need not coincide with $G_{\der,r}^F$. More precisely, the discrepancy is described by the exact sequence  
   $$
   1\to Z(G_{\sc,r})^F\to G_{\sc,r}^F\to G_{\der,r}^F\to H^1(F,Z(G_{\sc,r}))\to H^1(F,G_{\sc,r})=1,
   $$
   which arises from the Galois cohomology sequence associated to  
   $1\to Z(G_{\sc,r})\to G_{\sc,r}\to G_{\der,r}\to 1$.
\end{remark}

\begin{lemma}
    We have
    $$\bigsqcup_{\tau\in L^F_r/\pi_r((L_{\sc,r})^F)}\pi_r(X^{G_{\sc}}_{L_{\sc}, P_{\sc},r}) \dot \tau = X^{G}_{L, P, r},$$
    where $\dot\tau\in L^F_r$ is a fixed representative of $\tau\in L^F_r/\pi_r((L_{\sc,r})^F)$.
\end{lemma}
\begin{proof}
    We first prove the equality 
    $$\dot \tau \pi_r(X^{G_{\sc}}_{L_{\sc}, P_{\sc},r}) = \pi_r(X^{G_{\sc}}_{L_{\sc}, P_{\sc},r})\dot \tau.$$
    Since $L_r = L_{\der,r}Z(G)_r$, we may assume $\dot \tau \in L_{\der,r}$. Let $\tau_\sc$ be a lift of $\dot \tau$ under $\pi$. As $\tau_\sc \in L_{\sc,r}$ normalizes $V_r$, we have
    $$\tau_\sc X^{G_{\sc}}_{L_{\sc}, P_{\sc}, r} = X^{G_{\sc}}_{L_{\sc}, P_{\sc}, r} = X^{G_{\sc}}_{L_{\sc}, P_{\sc}, r}\tau_\sc.$$
    Taking images under $\pi_r$ yields the desired equality.
    
    Now let $x \in X^{G}_{L, P, r}$ and write $x^{-1}F(x) = v \in F(V_r)$. By Lang's theorem, there exists $u \in \pi_r(X^{G_{\sc}}_{L_{\sc}, P_{\sc}, r})$ such that $u^{-1}F(u) = v$. Thus there exists $g \in G^{F}_r$ satisfying $gu = x$. According to \ref{pisc}, we can find $\tau \in L^F_r/\pi_r((L_{\sc,r})^F)$ such that $g = \dot \tau g'$ with $g' \in \pi_r((G_{\sc,r})^F)$. Therefore,
    $$x = gu = \dot \tau g' u \in \dot \tau \pi_r(X^{G_{\sc}}_{L_{\sc}, P_{\sc}, r}),$$
    which gives the decomposition
    $$X^{G}_{L, P, r} = \bigcup_{\tau\in L^F_r/\pi_r((L_{\sc,r})^F)}\dot \tau\pi_r(X^{G_{\sc}}_{L_{\sc}, P_{\sc}, r}) = \bigcup_{\tau\in L^F_r/\pi_r((L_{\sc,r})^F)} \pi_r(X^{G_{\sc}}_{L_{\sc}, P_{\sc},r})\dot \tau.$$

    To prove that the union is disjoint, we suppose that there exist $\dot \tau \in L_r^F$ and $g_1, g_2 \in X^{G_{\sc}}_{L_{\sc}, P_{\sc}, r }$ such that $\pi_r(g_1)\dot \tau = \pi_r(g_2)$. It suffices to show that $\dot \tau \in \pi_r((L_{\sc,r})^F)$. Note that $\pi_r$ is the identity on $V_r$, so from 
    $$\pi_r(g_1)^{-1}F(\pi_r(g_1)) = \pi_r(g_2)^{-1}F(\pi_r(g_2)) \in F(V_r),$$ 
    we obtain 
    $$g_1^{-1}F(g_1) = g_2^{-1}F(g_2) \in F(V_r).$$ 
    Hence there exists $g \in (G_{\sc,r})^{F}$ such that $g_1g = g_2$. Therefore,
    $$\pi_r(g) = \dot \tau \in \pi_r((G_{\sc,r})^F \cap L_{\sc,r}) = \pi_r((L_{\sc,r})^F).$$
\end{proof}
\begin{proposition}\label{Re}
    Let $\theta$ be a character of $G^{F}_r$ trivial on $G^{F}_{\sc,r}$ and let $\psi\in \CR(L^F_r)$. Then
    $$R^{G}_{L, P, r}(\psi)\otimes \theta = R^{G}_{L, P, r}(\psi\otimes \theta|_{L^{F}_r}).$$
\end{proposition}

\begin{proof}
    For $g \in G_r^F$, we have
    $$R^{G}_{L, P, r}(\psi\otimes \theta|_{L^{F}_r})(g) = |L^{F}_r|^{-1}\sum_{l\in L_r^{F}}\mathrm{Tr}((g, l)|H^{*}_c(X^{G}_{L, P, r}))\psi(l^{-1})\theta(l^{-1}).$$
    
    By \ref{pisc}, there exists $\tau \in L^F_r/\pi_r((L_{\sc,r})^F)$ such that $g = \dot{\tau} g'$ with $g' \in \pi_r((G_{\sc,r})^F)$. For $l \in L^F_r$, if $\dot{\tau} l \notin \pi_r((L_{\sc,r})^F)$, then the action of $(g,l)$ on
    $$\bigsqcup_{\tau\in L^F_r/\pi_r((L_{\sc,r})^F)}\pi_r(X^{G_\sc}_{L_\sc, P_\sc, r}) \dot{\tau} = X^{G}_{L, P, r}$$
    maps each component to a different one. Hence
    $$\mathrm{Tr}((g, l)|H^{*}_c(X^{G}_{L, P,r})) = 0.$$
    
    If $\dot{\tau} l \in \pi_r((L_{\sc,r})^F)$, as $\theta$ is trivial on $G^{F}_{\sc,r}$, we have $\theta(g) = \theta(\dot{\tau}) = \theta(l^{-1})$. Thus we compute:
  \begin{align*}
        R^{G}_{L, P,r}(\psi\otimes \theta|_{L^{F}_r})(g) &= |L_r^{F}|^{-1}\sum_{l\in L_r^{F}}\mathrm{Tr}((g, l)|H^{*}_c(X^{G}_{L, P,r}))\psi(l^{-1})\theta(l^{-1}) \\
        &= R^{G}_{L, P,r}(\psi)(g)\theta(g) \\
        &= (R^{G}_{L, P,r}(\psi)\otimes \theta)(g)
    \end{align*}
\end{proof}
\begin{theorem}\cite[\S 2]{Chen25}\cite[\S 5]{Ch24}\label{Ell}
    Assume $T \subset G$ is elliptic over $k$ and fix $s < r$. Let $B\subset G$ be a Borel subgroup containing $T$. For any character $\theta:T^{F}_s\to \overline{\mathbb{Q}}^\times_\ell$, the following equality holds:
    $$R^{G}_{T, B,r}(\mathrm{Inf}^{T_r}_{T_s}(\theta)) = \mathrm{Inf}^{G_r}_{G_s}\left(R^{G}_{T, B, s}(\theta)\right).$$
\end{theorem}

\section{Scalar Product Formula for Generic Parahoric Lusztig Induction}\label{3}
In this section, we fix $r>0$ and a Borel subgroup $B$ containing $T$ whose unipotent radical is denoted by $U$. Let $L$ be an $F$-rational Levi subgroup standard with respect to $B$ containing $T$. Denote by $P = LB$ the parabolic subgroup with Levi decomposition $P = LV$.

We write
\begin{align*}
&\mathfrak{t} := T^r_r =  \mathrm{Lie}(T)^{r}_\mathbf{x}/\mathrm{Lie}(T)^{r+}_\mathbf{x} &\mathfrak{t}^* :=  \mathrm{Lie}^*(T)^{-r}_\mathbf{x}/\mathrm{Lie}^*(T)^{(-r)+}_\mathbf{x} \\
&\mathfrak{l} := L^r_r =  \mathrm{Lie}(L)^{r}_\mathbf{x}/\mathrm{Lie}(L)^{r+}_\mathbf{x} &\mathfrak{l}^* :=  \mathrm{Lie}^*(L)^{-r}_\mathbf{x}/\mathrm{Lie}^*(L)^{(-r)+}_\mathbf{x} \\
&\mathfrak{z} := Z(L)^r_r =  \mathrm{Lie}(Z(L))^r_\mathbf{x}/\mathrm{Lie}(Z(L))^{r+}_\mathbf{x}  
&\mathfrak{z}^* :=  \mathrm{Lie}^*(Z(L))^{-r}_\mathbf{x}/\mathrm{Lie}^*(Z(L))^{(-r)+}_\mathbf{x} \\
&\mathfrak{t}_{\mathrm{\der}}:=\mathrm{Lie}(L_{\der}\cap T)^r_\mathbf{x}/\mathrm{Lie}(L_{\der}\cap T)^{r+}_\mathbf{x}\\
&\mathfrak{t}^*_{\mathrm{der}}:=\mathrm{Lie}^\ast(L_{\der}\cap T)^{-r}_\mathbf{x}/\mathrm{Lie}^\ast(L_{\der}\cap T)^{(-r)+}_\mathbf{x}
\end{align*}
we have the decompositions $\mathfrak{t} \simeq \mathfrak{z} \times \mathfrak{t}_{\der}$ and $\mathfrak{t}^* \simeq \mathfrak{z}^* \times \mathfrak{t}^*_{\der}$
\subsection{Generic Characters and Howe factorization}
\begin{definition}\label{generic}($(L,G)$-generic character)
\begin{enumerate}
    \item An element $X\in (\mathfrak{z}^\ast)^F$ is called $(L,G)$-generic of depth $r$ if there is a lift $\tilde{X}\in \mathrm{Lie}^*(L)^{-r}_\mathbf{x}$ such that $\mathrm{val}(\tilde{X}(H_\alpha))=-r$ for all $\Phi(G,T) \smallsetminus \Phi(L,T)$, where $H_\alpha:=d\alpha^{\vee}(1)\in \mathrm{Lie}(L)$ with $d\alpha^{\vee}$ the derivative of the coroot $\alpha^{\vee}$ (i.e., condition GE1 in \cite[\S 8]{Yu}).
    \item Fix an additive character $\phi:k\to {\overline{\mathbb{Q}}^\times_\ell}$, a character $\psi$ of $L(k)$ is called $(L,G)$-generic (of depth $r$) if $\psi$ is trivial on $L^{r+}_\mathbf{x}(k)$ and its restriction to $\mathfrak{l}^F$ is given by $\phi \circ X$ for some $(L,G)$-generic element $X$ of depth $r$.
    \end{enumerate}
\end{definition}

\begin{lemma}\label{genlemma}
    \cite[Lemma 3.6.8]{Kal} Let $\psi:L(k)\to {\overline{\mathbb{Q}}^\times_\ell}$ be a character of depth $r$, trivial on $L_\sc(k)$. Then $\psi$ is $(L,G)$-generic if and only if for all $\alpha \in \Phi(G,T) \smallsetminus \Phi(L,T)$, we have $\psi|_{N_{E/k}(\alpha^\vee(E_r^\times))} \neq \mathrm{triv}$, where $E$ is a splitting field of $T$. 
\end{lemma}
\begin{remark}
    Since $p$ satisfies condition $(\ast)$, an element $X\in (\mathfrak{z}^\ast)^F$ satisfies GE1 automatically satisfies condition GE2 in \cite[\S 8]{Yu} by \cite[Lemma 8.1]{Yu}.
\end{remark}

\begin{definition}\label{generic2}
\begin{enumerate}

    \item We say that $\rho\in \mathcal{R}(L_r^F)$ is $(L,G)$-generic if $\rho|_{\mathfrak{l}^F}$ is the restriction of a sum of $(L,G)$-generic characters of depth $r$.
    \item A character $\theta$ of $T^F_r$ is called $(L,G)$-generic if $\theta|_{\mathfrak{t}^F}$ is the restriction of an $(L,G)$-generic character trivial on $L_{\sc}(k)$.
\end{enumerate}
\end{definition}

\begin{definition}\label{Howe}[Howe factorization \cite[\S 3]{Kal}]
  Let \( G^{-1} = T \) and $\theta$ be a character of \( T(k)\). A Howe factorization of \( (\theta, T) \) is a sequence of $k$-rational Levi subgroups  
  \[
    G^0 \subset G^1 \subset \dots \subset G^d = G,
  \]  
  along with characters \( \phi_i \colon G^i(k) \to \overline{\mathbb{Q}}^\times_\ell \) for \( i = -1, 0, \dots, d \), satisfying the following properties:  
  \begin{enumerate}
    \item \( \theta = \prod_{i=-1}^d \phi_i|_{T(k)} \).  
    \item For \( 0 \leq i \leq d \), the character \( \phi_i \) is trivial on \( G^i_{\sc}(k) \).  
    \item For \( 0 \leq i < d \), the character \( \phi_i \) has depth \( r_i \) and is \((G^i, G^{i+1})\)-generic.  
      For \( i = d \), $\phi_d$ is trivial if \( r_d = r_{d-1} \) and has depth $r_d$ otherwise. 
      For \( i = -1 \), $\phi_{-1}$ is trivial if \( G^0 = T \) and  
      otherwise satisfies \( \phi_{-1}|_{T(k)_{0+}} = 1 \)  
  \end{enumerate}
\end{definition}

\begin{theorem}\cite[Proposition 3.6.7]{Kal}\label{Howexist}
    Any character $\theta$ of $T(k)$ admits a Howe factorization $(\phi_{-1},\dots,\phi_d)$.
\end{theorem}

The proof of the above theorem relies mainly on the following lemma.
\begin{lemma}\cite[Lemma 3.6.9]{Kal}\label{lemma1}
    Let $\theta$ be a character of $T(k)$ of depth $r$. If for all $\alpha\in \Phi(L,T)$ we have $\theta\circ N_{E,k}\circ {\alpha^{\vee}}|_{E^\times_r}=1$, then there exists a character $\phi:L(k)\to \overline{\mathbb{Q}}^\times_\ell$ of depth $r$ that is trivial on $L_{\sc}(k)$ and satisfies $\phi|_{\mathfrak{t}^F}=\theta|_{\mathfrak{t}^F}$.
\end{lemma}

\begin{corollary}\label{ellgeneric}
    Let $T\subset G$ be elliptic over $k$ and $\theta$ be an $(L,G)$-generic character of $T^F_r$, then $R^L_{T,B\cap L,r}(\theta)$ is $(L,G)$-generic.
\end{corollary}
\begin{proof}
 Since $\theta$ is $(L,G)$-generic, Lemma~\ref{lemma1} yields a character $\phi\colon L^F_r \to {\overline{\mathbb{Q}}^\times_\ell}$ that is trivial on $L^F_{\mathrm{\sc},r}$. As $\phi|_{\mathfrak{t}^F} = \theta|_{\mathfrak{t}^F}$, the character $\theta' := \theta \otimes (\phi|_{T^F_r})^{-1}$ has depth $r' < r$. By Proposition~\ref{Re} and Theorem~\ref{Ell}, we derive:
\[
R^L_{T,B\cap L,r}(\theta) = R^L_{T,B\cap L,r}(\theta' \otimes \phi|_{T^F_r}) = R^L_{T,B\cap L,r}(\theta') \otimes \phi = \mathrm{Inf}^{L^F_r}_{L_{r'}^F}R^L_{T,B\cap L,r'}(\theta') \otimes \phi.
\]
which concludes the proof.
\end{proof}
\subsection{Deep level Bruhat decomposition}




For $M \in \{L, T\}$, let $Q \in \{P, B\}$ be their corresponding parabolic subgroups with unipotent radicals $N \in \{V, U\}$ respectively. 

We denote the ``Weyl group" $W_{G_r}:=N_{G_r}(T_r)/T_r$. The natural projection $G_r \to G_0$ induces a map 
\begin{equation}\label{Weyl}
    N_{G_r}(T_r)/T_r\to N_{G_0}(T_0)/T_0,
\end{equation}
which is an isomorphism. Hence, under the identification \eqref{Weyl}, we have a canonical bijection
$$W_{M_r} \backslash W_{G_r}/W_{L_r}\cong W_{M_0} \backslash W_{G_0}/W_{L_0}$$

We have a deep level Bruhat decomposition for $G_r$:
\begin{proposition}
The group $G_r$ admits a natural decomposition
\[
G_r = \bigsqcup_{w \in W_{M_r} \backslash W_{G_r}/W_{L_r}} G_{r,w}
\]
where
\begin{itemize}
    \item $G_{r,w} := Q_r K_w \dot{w} P_r$
    \item $K_w := (N^- \cap {^{\dot{w}} V^-})^{0+}_{r}$, where $N^-$ and $V^-$ are the opposite unipotent radicals of $N$ and $V$ respectively.
    \item $\dot{w}$ represents a chosen lift in $N_{G_r}(T_r) \subset G_r$ for each $w \in W_{M_r} \backslash W_{G_r}/W_{L_r}$
\end{itemize}
Indeed, this decomposition is induced by the pullback of the Bruhat decomposition of $G_0$ under the natural projection $ G_r \to G_0$.
\end{proposition}

\begin{proof}
We follow the proof from \cite[Lemma 2.3]{Sta09} and \cite[Lemma 2.7]{CI21}. 

The Bruhat decomposition 
\[
G_0 = \bigsqcup_{w \in W_{M_0} \backslash W_{G_0}/W_{L_0}} Q_0 w P_0
\]
implies that the preimages of each parabolic Bruhat cell under the projection $f: G_r \to G_0$ naturally decompose $G_r$ into a union of locally closed subsets. It suffices to show that the sets $G_{r,w}$ correspond precisely to these preimages $f^{-1}(Q_0 \dot{w} P_0)$.

We compute:
\begin{align*}
f^{-1}(Q_0 w P_0) &= Q_r \dot{w} P_r G^{0+}_{r} \\
&= Q_r \dot{w} (G^{0+}_{r} \cap V^{-}_r) P_r \\
&= Q_r \dot{w} (G^{0+}_{r} \cap V^{-}_r) \dot{w}^{-1} \dot{w} P_r \\
&= Q_r (N^{-}_r \cap \dot{w} (G^{0+}_{r} \cap V^{-}_r) \dot{w}^{-1}) \dot{w} P_r \\
&= Q_r (N_r^{-} \cap \dot{w} V_r^{-} \dot{w}^{-1})^{0+}_{r} \dot{w} P_r \\
&= Q_r K_w \dot{w} P_r
\end{align*}
\end{proof}

\begin{remark}
   For each $w\in W_{G_r}$ we can find a lift $\dot w\in N_{G_r}(T_r)$ such that $\dot w$
is fixed by $F^n$, where $n$ is the integer such that $T$ is $F^n$-split. Throughout this paper, for each $w\in W_{M_r} \backslash W_{G_r}/W_{L_r}$, we may choose its lift $\dot w\in N_{G_r}(T_r)$ such that $F^n(\dot w)=\dot w$.
\end{remark}


\subsection{Steinberg type variety}
We study the cohomology of the fiber product
$X^{G}_{M,Q,r} \times_{G^{F}_r} X^{G}_{L,P,r}$,
where $(m,g) \in M^{F}_r \times G^{F}_r$ acts on $X^{G}_{M,Q,r}$ via
$x \mapsto g^{-1}x m\i$.

We have an $(M^{F}_r\times L^{F}_r)$-equivariant isomorphism
    \[X^{G}_{M, Q,r}\times_{G^{F}_r} X^{G}_{L, P,r}\to \Sigma^{M}:=\{(x, x', y)\in F(N_r)\times F(V_r)\times G_r:xF(y)=yx'\}\]
\begin{equation}\label{iso1}
(g, g')\mapsto (g^{-1}F(g), g'^{-1}F(g'), g^{-1}g'),
\end{equation}
where $M^{F}_r\times L^{F}_r$ acts on $\Sigma^{M}$ given by
$$(m, m'):(x, x',y)\mapsto (mxm^{-1}, m'^{-1}xm', mym').$$
For each $w \in W_{M_r} \backslash W_{G_r}/W_{L_r}$, set
\begin{equation*}
\Sigma^{M}_{w} := \{(x, x', y) \in \Sigma^{M} \mid y \in G_{r,w}\},
 \end{equation*}
which is $(M^F_r \times L^F_r)$-stable.

Set
\begin{equation*}
    \widehat{\Sigma}^{M}_w := \left\{
    \begin{gathered}(x, x', u, u', z, \mu)\in F(N_r)\times F(V_r)\times N_r\times V_r \times K_w\times {^{\dot{w}\i} M_r}L_r\\
        xF(z\dot w\mu)=uz\dot w\mu u'x'
    \end{gathered}
     \right\}
\end{equation*}

with  $M^{F}_r\times L^{F}_r$ action:
\begin{equation}
    \label{Laction}
    (m, m'):(x, x', u, u', z, \mu)\mapsto (^mx, ^{m'^{-1}}x', ^mu, ^{m'^{-1}}u', ^mz, ^{{\dot w}^{-1}}m\mu m')
\end{equation}
\begin{lemma}
There is an $(M^F_r \times L^{F}_r)$-equivariant affine fibration 
    \[
    \widehat{\Sigma}^{M}_w \rightarrow \Sigma^{M}_w
    \]
    \begin{equation}\label{fib}
    (x, x', u, u', z, \mu) \mapsto (xF(u)^{-1}, x'F(u'), uz\dot{w}\mu u')
     \end{equation}
Hence, the cohomology $H_c^*(\widehat{\Sigma}^M_w)$ is isomorphic to $H_c^*(\Sigma^M_w)$ as a $(M^F_r \times L^F_r)$-module. 
\begin{proof}
    The map \ref{fib} is composed of an isomorphism of $ \widehat{\Sigma}^{M}_w$
    \[(x,x',u,u',z,\mu)\mapsto (xF(u)\i,x'F(u'),u,u',z\mu)\]
    and an affine fibration
    \[(x,x',u,u',z,\mu)\mapsto (x,x',uz\dot{w}\mu u')\]
\end{proof}
\end{lemma}

Define the following $(M^F_r \times L^F_r)$-stable subvarieties of $\widehat{\Sigma}^{L}_w$:
\begin{align*}
\widehat{\Sigma}^{\prime M}_w := \{(x, x', u, u', z, \mu) \in \widehat{\Sigma}^{M}_w \mid z \neq 1\}
\\
\widehat{\Sigma}^{\prime\prime M}_w := \{(x, x', u, u', z, \mu) \in \widehat{\Sigma}^{M}_w \mid z = 1\}
\end{align*}
We denote by $\Sigma^{\prime M}_w$ and $\Sigma^{\prime\prime M}_w$ the images of $\widehat{\Sigma}^{\prime M}_w$ and $\widehat{\Sigma}^{\prime\prime M}_w$ under the fibration \ref{fib}, respectively.

\begin{proposition}\label{composeTLL}
We establish a $(T^F_r \times L^F_r)$-equivariant isomorphism which holds for all $w \in W_{L_r} \backslash W_{G_r}/W_{L_r}$
\[
X^{L}_{T, B\cap L,r} \times_{L^F_r} \Sigma^{L}_w  \cong \bigsqcup_{w'} \Sigma^{T}_{w'}
\]
\[(m,(x,x',y))\mapsto (m x F(m)\i,x',m y)\]
where $w'$ ranges over all elements in $W_{G_r}/W_{L_r}$ with the same image as $w$ under the natural quotient map. 
\end{proposition}
\begin{proof}
    Applying \ref{trans} and \ref{iso1}, we obtain the following isomorphism
\[
\sigma:X^{L}_{T, B\cap L ,r} \times_{L^F_r} \Sigma^{L}  \cong  \Sigma^{T}
\]
\[(m,(x,x',y))\mapsto (m x F(m)\i,x',m y).\]
One checks that 
\begin{align*}
    \sigma(X^{L}_{T, B\cap L,r} \times_{L^F_r} \Sigma^{L}_w ) &=\{(x,x',y)\in \Sigma^{T}; y\in P_rK_w\dot{w}P_r\}\\
    &=\{(x,x',y)\in \Sigma^{T}; f(y)\in P_0\dot{w}P_0\}\\
    &=\bigsqcup_{w'}\{(x,x',y)\in \Sigma^{T}; f(y)\in B_0w'P_0\}\\
    &=\bigsqcup_{w'}\{(x,x',y)\in \Sigma^{T}; y\in B_rK_{w'}\dot{w}'P_r\}\\
    &=\bigsqcup_{w'} \Sigma^{T}_{w'}.
\end{align*}
where $f:G_r\to G_0$ is the natural projection map.
\end{proof}
\subsection{Some Vanishing Results}
\begin{proposition}\label{vansLL}
Let $\psi\in \mathcal{R}(L_{r}^F)$ be $(L,G)$-generic, choose $  w \in W_{L_r} \backslash W_{G_r}/W_{L_r}$ such that its lift $\dot w$ normalizes $L_r$. Then for all $i\geq 0$,
    $$H_c^{i}(\widehat{\Sigma}'^{L}_w)\otimes_{L^{F}_r}\psi=0.$$
    In other words, $H_c^{i}(\widehat{\Sigma}'^{L}_w)$ contains no subspace on which $L^F_r$ acts via $\psi$.
\end{proposition}

\subsubsection{Proof of \ref{vansLL}}
The proof is a generalization of the arguments used in [\cite{Lu04,Sta09,CI21,Ch24}].
Following \cite{Lu04,CI21}, we have a stratification
  $$K_{w} \smallsetminus \{1\}=(V^-_r\cap \dot{w}V^-_r\dot{w}\i)^{0+}_r\smallsetminus \{1\} = \bigsqcup_{1 \leq a \leq r} \bigsqcup_{I \in \mathcal{X}} K_{w}^{a,I},
$$
where $\mathcal X$ is the set of nonempty subsets of $\{\beta \in \Phi(G,T) \smallsetminus \Phi(L,T) : U_{\beta,r}\subset V^-_r\cap \dot{w}V^-_r\dot{w}\i\}$, where $U_\beta$ is the root subgroup of $G$ corresponding to $\beta \in \Phi(G,T)$. 

The stratification gives rise to a partition $\widehat\Sigma'^{L}_w = \bigsqcup_{a,I} \widehat \Sigma_w^{\prime,L, a, I}$, where
$$\widehat \Sigma_w^{\prime,L, a, I}=\{(x,x',u,u',z,\mu)\in \widehat \Sigma_w^{\prime,L};z\in K_{w,r}^{a,I}\}$$
Fix a parameter $(a,I)$ with $1 \leq a \leq r$ and $I \in \mathcal{X}$. Consider the map
$$\widehat{\Sigma}_w^{\prime,L,T, a, I} \to  L_0, \quad (x, x', u, u', z, \mu) \mapsto \overline{\mu}$$
where $ L_r\to   L_0, \mu \mapsto \bar\mu$ is the natural projection. 

Let $\widehat{\Sigma}_{w,\bar\mu}^{\prime,L, a, I}$ denote the fiber over $ L_0$.

\begin{lemma}
     Let $\alpha \in \Phi(G,T)$ be such that $-\alpha \in I$. Then
  $\widehat \Sigma_{w,\bar \mu}^{\prime,L, a, I}$ admits an action of the algebraic group
  $$
    \mathcal H := \{l \in L^{r}_{r} : lF(l)^{-1}  \in \mu^{-1} \dot w^{-1} T_{r}^{\alpha,r} \dot w \mu\}.
  $$
  where $T^\alpha$ is a rank 1 subtorus of $T$ in group generated by $U_\alpha$ and $U_{-\alpha}$. Moreover, when restricted to the subgroup $\mathcal{H} \cap L^F_r$ this action coincides with \ref{Laction}.
\end{lemma}
\begin{proof}
  The argument is similar to \cite[p.7]{Lu04} and \cite[Lemma~4.10]{Ch24}. By \cite[Lemma~4.10]{CI21} and the fact that both $\xi$ and $z$ are normalized by $L_r={}^{\dot{w}}L_r$, the construction of $K_{w,r}^{a,I}$ guarantees the inclusion
$
[U^{r-a}_{\alpha,r}, K_{w,r}^{a,I}] \subset T_{r}^{\alpha,r}({}^{\dot{w}}V^r_r \cap V^r_r).
$
Choose any $z \in K_{w,r}^{a,I}$. For any $\xi \in U^{r-a}_{\alpha,r}$, we write
     \begin{equation*}
    [\xi^{-1},z^{-1}] = \tau_{\xi,z} \cdot \omega_{\xi,z}, \qquad \text{where $\tau_{\xi,z} \in T^{\alpha,r}_{r}$ and $\omega_{\xi,z} \in {^{\dot{w}}} V^r_r\cap V^r_r$.}
  \end{equation*}
  This defines a surjective map $\lambda_z:U^{r-a}_{\alpha,r}\to T_{r}^{\alpha,r}$; $\xi \mapsto \tau_{\xi,z}$. Fix a section $\lambda\i_z: T_{r}^{\alpha,r}\to U^{r-a}_{\alpha,r}$ of $\lambda_z$

  For $l\in \mathcal{H}$, define $f_{l}:\widehat \Sigma_{w,\bar \mu}^{\prime,L, a, I}\to \widehat \Sigma_{w,\bar \mu}^{\prime,L, a, I}$ as follows:
  $$f_{l}(x,x',u,u',z,\mu) = \left(x F(\xi), \hat{x}', u, F(l)^{-1}u'F(l), z, \mu F(l)\right)$$
  where
  \[
    \xi = \lambda^{-1}_z(\dot w \mu l F(l)^{-1} \mu^{-1} \dot w^{-1}) \in U^{r-a}_{\alpha,r} \subset  {^{\dot{w}}} V^r_r\cap V_r
  \]
  and  $\hat{x}'\in G_r$ is defined by the condition that
  \[x F(\xi) F(z) F(\dot w) F(\mu) F^2(l) = u z \dot w \mu u' F(l) \hat x'\]
  In order for this to be well defined, it is enough to check that $\hat{x}'\in F(V_r)$.
  First from the definition of $\xi$, we have 
  \begin{equation*}
      \mu^{-1} \dot w^{-1} \tau_{\xi,z} \dot w \mu = l F(l)^{-1},\qquad \text{where $\tau_{\xi,z}\in T^{\alpha,r}_r$}
      \end{equation*}
  or that $$F(\mu)^{-1} F(\dot w)^{-1} F(\tau_{\xi,z}) F(\dot w) F(\mu F(l))=F(l),$$
  Since $x' \in F(V_r)$ and $^{F(\dot w^{-1})} F(\omega_{\xi,z}) \in F(V_r\cap {^{\dot{w}\i}} V^r_r)$,$^{F(\dot w^{-1})} F(\xi) \in F(V_r\cap {^{\dot{w}\i}} V^r_r)$. The last two elements are normalized by $L_r$, this implies
  \[
    x' F(\mu)^{-1} F(\dot{w})^{-1} F(\xi) F(\tau_{\xi, z}) F(\omega_{\xi,z}) F(\dot{w} \mu F(l)) \in F(l) F(V_r).
    \]
By definition, we have $xF(z) = u z \dot w \mu u' x' F(\mu)^{-1} F(\dot w)^{-1}$, the previous statement holds if and only if
  \[
    x F(z) F(\xi) F(\tau_{\xi,z}) F(\omega_{\xi,z}) F(\dot{w} \mu F(l)) \in u z \dot{w} \mu u' F(l) F(V_r) 
  \]
  Substituting the relation $\xi z = z \xi \tau_{\xi, z} \omega_{\xi, z}$ into the equation above, we obtain
 \[
    x F(\xi) F(z \dot{w} \mu F(l)) \in u z \dot{w} \mu u' F(l) F(V_r)
 \]
 which implies  $\hat{x}'\in F(V_r)$. Clearly, $f_{l}\circ f_{l'}=f_{l'l}$, hence this construction defines a right $\mathcal{H}$-action on $\widehat \Sigma_{w,\bar \mu}^{\prime,L, a, I}$.
\end{proof}
Let \( n \geqslant 1 \) be an integer satisfying
\[
F^n\bigl( \mu^{-1}\dot{w}^{-1} T^{\alpha,r}_{r} \dot{w}\mu \bigr) 
= \mu^{-1}\dot{w}^{-1} T^{\alpha,r}_{r} \dot{w}\mu,
\]
Consider the norm map 
\begin{align*}    
\mathcal{N}^{F^n}_F \colon & \mu^{-1}\dot{w}^{-1} T^{\alpha,r}_{r} \dot{w}\mu \to \mathcal{H}\\
&t \mapsto t F(t) F^2(t) \cdots F^{n-1}(t).
\end{align*}
To verify that the image lies in $\mathcal{H}$, note that \(G^r_{r}\) is abelian and compute:
\[
\mathcal{N}^{F^n}_F(t) F\left(\mathcal{N}^{F^n}_F(t)\right)^{-1} 
= \mathcal{N}^{F^n}_F\left(t F(t)^{-1}\right) 
= t F^n(t)^{-1} \in \mu^{-1}\dot{w}^{-1} T^{\alpha,r}_{r} \dot{w}\mu.
\]
Since norm map $\mathcal{N}^{F^n}_F$ is an isogeny of $G^r_r$, $\mathcal{N}^{F^n}_F(\mu^{-1}\dot{w}^{-1} T^{\alpha,r}_{r} \dot{w}\mu)$ is connected, it is contained in the identity component $\mathcal{H}^0$ of $\mathcal{H}$. The following lemma is proved.

\begin{lemma}
    The intersection $\mathcal{H}^0\cap (L^r_r)^F$ contains $\mathcal{N}^{F^n}_F(\mu^{-1}\dot{w}^{-1} T^{\alpha,r}_{r} \dot{w}\mu)$.
\end{lemma}

The action of the connected algebraic group $\mathcal{H}^0$ on $H^i_c(\hat{\Sigma}^{\prime L a, I}_{w, \bar{\mu}})$ must be trivial. Hence, $\mathcal{N}^{F^n}_F ((\mu^{-1}\dot{w}^{-1} T^{\alpha,r}_{r} \dot{w}\mu)^{F^n})$ also acts trivially.

However, as $\dot{w}^{-1}\alpha \notin \Phi(L, T)$, Lemma~\ref{genlemma} implies that $\psi \circ \mathcal{N}^{F^n}_F$ is nontrivial on $\mathcal{N}^{F^n}_F ((\mu^{-1}\dot{w}^{-1} T^{\alpha}_{r:r+} \dot{w}\mu)^{F^n})$. Consequently, $H^i_c(\hat{\Sigma}^{\prime L, a, I}_{w, \bar{\mu}})$ contains no subspace where $\mathcal{N}^{F^n}_F ((\mu^{-1}\dot{w}^{-1} T^{\alpha}_{r:r+} \dot{w}\mu)^{F^n)}$ acts via $\psi$. 

We therefore conclude that 
\[
H_c^{i}(\hat{\Sigma}^{\prime L}_w) \otimes_{L^{F}_r} \psi = 0,
\] 
which completes the proof of Proposition \ref{vansLL}.

\begin{proposition}\label{vansTL}\cite[Proposition 4.8]{Ch24}
    Let $\psi\in \CR(L_{r}^F)$ be $(L,G)$-generic. For $w \in  W_{G_r}/W_{L_r}$ and all $i\geq 0$,
    $$H_c^{i}(\widehat{\Sigma}'^{T}_w)\otimes_{L^{F}_r}\psi=0$$
\end{proposition}

\begin{proposition}\label{vansw}
    Let $\theta:T^F_r\to \ov\BQ^\times_\ell$ be an $(L,G)$-generic character and $\psi$ be an $(L,G)$-generic irreducible constituent of $R^L_{T,B\cap L,r}(\theta)$. Then for $w \in W_{G_r}/W_{L_r}$ and all $i\geq 0$,
     $$\overline\theta^{op}\otimes_{T^{F}_r}H_c^{i}({\Sigma}^{T}_w)\otimes_{L^{F}_r}\psi\neq0$$
    holds only if $ w=1$.
\end{proposition}
\begin{proof}
    Since $\psi$ is an $(L,G)$-generic irreducible constituent of $R^L_{T,B\cap L,r}(\theta)$, 
    there exists $j\in \mathbb{Z}$ such that the space $\overline\theta^{op}\otimes_{T^{F}_r}H_c^{i}({\Sigma}^{T}_w)\otimes_{L^{F}_r}\psi$ embeds as a subspace of $$\overline\theta^{op}\otimes_{T^F_r}H_c^{i}({\Sigma}^{T}_w)\otimes_{L^{F}_r}H^j(X^L_{T,B\cap L,r})\otimes_{T^{F}_r}\theta$$
    hence a subspace of $$\overline\theta^{op}\otimes_{T^F_r}H_c^{i+j}(\widehat{\Sigma}^{''T}_w\times_{L^{F}_r}X^L_{T,B\cap L,r})\otimes_{T^{F}_r}\theta$$
    via Proposition~\ref{composeTLL} 


Then a direct verification shows that equation \ref{Laction} also defines an action of
$$S_w:=\{(t,t')\in T_r \times Z(L_r); t^{\i} F(t)=^{F(\dot w)}(t' F(t'^{\i}))\}$$
on $\widehat{\Sigma}^{''T}_w$. Since the action of $S_w$ commutes with that of $L^F_r$ on $\widehat{\Sigma}^{''T}_w$, it induces an action of $S_w$ on $\widehat{\Sigma}^{''T}_w\times_{L^F_r}X^L_{T,B\cap L,r}$.

Define
\[
Z'_w=\{ (\CN^{F^n}_F({}^{F(\dot w)}t),\,\CN^{F^n}_F(t^{-1})) \mid t\in \mathfrak{z} \}\subset  T_r \times Z(L_r).
\]
where $n$ is the integer such that $T$ is $F^n$-split. A straightforward calculation shows that $Z'_w$ is a connected subgroup of $S_w$, and consequently $Z'_w \subset S^0_w$.  
Moreover, define
\[
Z_w=\{ (\CN^{F^n}_F({}^{F(\dot w)}t),\,\CN^{F^n}_F(t^{-1})) \mid t\in \mathfrak{z}^{F^n} \}\subset  T_r \times Z(L_r).
\]
Then $Z_w$ is also a subgroup of $S^0_w$.

This implies the action of $Z_w$ must be trivial, yielding the key condition:
\begin{equation}\label{Ad}
(\overline\theta\circ \CN^{F^n}_F\circ (-)^{-1})|_{\mathfrak{z}^{F^n}}=(\theta\circ \CN^{F^n}_F)|_{\mathfrak{z}^{F^n}}=(\theta\circ \CN^{F^n}_F\circ Ad (\dot w))|_{\mathfrak{z}^{F^n}}
\end{equation}

Since $\theta$ is an $(L,G)$-generic character on $\mathfrak{t}^F$, 
by Definition~\ref{generic2}, there is an $(L,G)$-generic character $\theta'$ on $\mathfrak{l}^{F}$ such that $\theta=\theta'|_{\mathfrak{t}^{F}}$. Moreover, the composition $\theta'\circ \CN^{F^n}_F$ yields an $(L,G)$-generic character on $\mathfrak{l}^{F^n}$ via Lemma~\ref{genlemma}. By Definition~\ref{generic}, we can identify $\theta'\circ \CN^{F^n}_F$ with an $(L,G)$-generic element $X_n$ in $(\mathfrak{z}^{\ast})^{F^n}$ such that $\phi\circ X_n|_{\mathfrak{l}^{F^n}}=\theta'\circ \CN^{F^n}_F$.
Note that $X_n$ satisfies condition GE2 in \cite[\S 8]{Yu} that is $\mathrm{Stab}_{W_{G_r}}(X_n)=W_{L_r}$. 

The equality \ref{Ad} then translates to the adjoint condition 
\begin{equation}\label{Ad2}
Ad(\dot w)X_n=t'+ X_n\in (\mathfrak{t}^{\ast})^{F^n} 
\end{equation}
for some $t'\in (\mathfrak{t}^{\ast}_{\der})^{F^n}$. Set $\mathfrak{g}^\ast:=\mathrm{Lie}^\ast(T)^{-r}_{\bf x}/\mathrm{Lie}^\ast(T)^{(-r)+}_{\bf x}$, we have $\iota:(\mathfrak{g}^\ast)^{F^n} \hookrightarrow \mathfrak{gl}_{N}(\mathbb{F}_{q^n})$ such that $(\mathfrak{t}^{\ast})^ {F^n}=(\mathfrak{t}^{\ast}_{\der})^{F^n}\times (\mathfrak{z}^{\ast})^{F^n}$ diagonally into $\mathfrak{gl}_{N}(\mathbb{F}_{q^n})$. Moreover, by composing the embedding \(\iota\) with a suitable matrix conjugation, 
we have 
$\iota((\mathfrak{t}^{\ast}_{\der})^{F^n})\subset \mathfrak{gl}_{M}(\mathbb{F}_{q^n})$, $\iota((\mathfrak{z}^{\ast})^{F^n})\subset \mathfrak{gl}_{N-M}(\mathbb{F}_{q^n})$,
so the action of \(\operatorname{Ad}(\dot{w})\) on \(\iota((\mathfrak{t}^{\ast})^{F^n})\subset \mathfrak{gl}_{N}(\mathbb{F}_{q^n})\) 
can be identified with the corresponding matrix conjugation. 
Since $\iota(X_n)$ and $\iota(X_n+t')$ are conjugate if and only if $t'=0$, by condition \ref{Ad2}, we must have $t'=0$. Consequently, $\dot w\in \mathrm{Stab}_{W_{G_r}}(X_n)=W_{L_r}$.
\end{proof}

\subsection{Main Result}
\begin{theorem}\label{thm}
Let $\psi, \psi'\in \CR(L^F_r)$ be $(L,G)$-generic, both occurring as constituents of $R^{L}_{T, B\cap L,r}(\theta)$, where $\theta$ is an $(L,G)$-generic character of $T^F_r$. Then the following equality holds:
$$\big\langle R^{G}_{L, P,r}(\psi'), R^{G}_{L, P,r}(\psi)\big\rangle_{G^F_r} = \big\langle \psi', \psi\big\rangle_{L^F_r}$$
\end{theorem}

\begin{proof}
    It suffices to treat the case where $\psi,\psi'$ are irreducible.  Lemma~\ref{bar} implies that $\overline{\psi}'$ is a constituent of $R^{L}_{T, B \cap L,r} (\overline{\theta})$. 
    Therefore, for any $i \in \mathbb{Z}$, there exists $j \in \mathbb{Z}$ such that
$
\overline{\psi}'^{\mathrm{op}} \otimes_{L^{F}_r} H^i_c(\Sigma^{L}_w) \otimes_{L^{F}_r} \psi
$
 embeds into
\[
\overline{\theta}^{\mathrm{op}} \otimes_{T^{F}_r} H^j_c(X^{L}_{T, B \cap L,r}) \otimes_{L^{F}_r} H^i_c(\Sigma^{L}_w) \otimes_{L^{F}_r}  \psi.
\]
By Proposition~\ref{composeTLL}, this is isomorphic to a submodule of $$\overline\theta^{op}\otimes_{T^F_r} H_c^{i+j}(\bigcup_{w'}\Sigma^{T}_{w'})\otimes_{L^F_r}\psi$$
where $w'$ runs over the elements of $W_{G_r}/W_{L_r}$ having the same image in $W_{L_r} \backslash W_{G_r}/W_{L_r}$ as $w$.
 
Proposition~\ref{vansw} establishes that $\overline\theta^{op}\otimes_{T^F_r} H_c^{i+j}(\Sigma^{T}_{w'})\otimes_{L^F_r}\psi$ vanishes unless $w'=1$. Combining this with Proposition~\ref{vansLL}, we obtain
$$
\overline{\psi}'^{op} \otimes_{L^{F}_r} H^i_c(\Sigma^{L}) \otimes_{L^{F}_r} \psi=\overline{\psi}'^{op} \otimes_{L^{F}_r} H^i_c(\Sigma^{L}_1) \otimes_{L^{F}_r} \psi=\overline{\psi}'^{op} \otimes_{L^{F}_r} H^i_c(\Sigma^{''L}_1) \otimes_{L^{F}_r} \psi.
$$

Recall the variety $\Sigma''^L_1=\{(x, x', y)\in F(V_r)\times F(V_r)\times P_r:xF(y)=yx'\}$, and consider the following composition $$\pi:\Sigma''^L_1\to P_r\to L_r$$
    $$(x, x', y)\mapsto y\mapsto \bar y$$
    Since $y=xF(y)x'^{-1}=F(y)(F(y)^{-1}xF(y))x'^{-1}\in P_r\cap F(P_r)=L(V_r\cap F(V_r))$, we see that $y\in L^F(V_r\cap F(V_r))$; so $\Sigma''^L_1$ projects to $L^F_r$ under the above composition with fibers all isomorphic to $F(V_r) \times(V_r\cap F(V_r))$. Thus, up to shift by $2d=2\mathrm{dim}(V_r\cap F(V_r))+2\mathrm{dim}(V_r)$, the cohomology is that of the discrete variety $L^F_r$, thus
    $$\overline\psi'^{op}\otimes_{L^{F}_r}H^i_c(\Sigma^L)\otimes_{L^{F}_r}\psi=\overline\psi'^{op}\otimes_{L^{F}_r}H^i_c(\Sigma''^L_1)\otimes_{L^{F}_r}\psi=$$
  \begin{equation}\label{key}\begin{cases} 
   \overline\psi'^{op}\otimes_{L^{F}_r}L^{F}_r\otimes_{L^{F}_r}\psi &  i=2d\\
 0 & otherwise
\end{cases}  
\end{equation}
and the dimension of $\overline\psi'^{op}\otimes_{L^{F}_r}L^{F}_r\otimes_{L^{F}_r}\psi $ is $\big\langle \psi', \psi\big\rangle_{L^F_r}$.

The scalar product $\big\langle R^{G}_{L, P,r}(\psi'), R^{G}_{L, P,r}(\psi)\big\rangle_{G^F_r}$ is equal to the alternating sum of the dimension of $$\overline\psi'^{op}\otimes_{L^{F}_r}H^i_c(\Sigma^L)\otimes_{L^{F}_r}\psi,$$
hence $\big\langle R^{G}_{L, P,r}(\psi'), R^{G}_{L, P,r}(\psi)\big\rangle_{G^F_r}=\big\langle \psi', \psi\big\rangle_{L^F_r}$
\end{proof}

\begin{theorem}\label{Ellthm}
Let $T$ be elliptic over $k$ and let $\theta$ be an $(L,G)$-generic character of $T^F_r$. For any two constituents $\psi, \psi'$ of
$R^{L}_{T, B\cap L,r}(\theta)$:
$$\big\langle R^{G}_{L, P,r}(\psi'), R^{G}_{L, P,r}(\psi)\big\rangle_{G^F_r} = \big\langle \psi', \psi\big\rangle_{L^F_r}.$$
\end{theorem}
\begin{proof}
    By Lemma~\ref{ellgeneric}, the virtual representation $\psi, \psi'$ satisfy the hypotheses in Theorem~\ref{thm}.
\end{proof}

\begin{corollary}
Let $\psi$ be an irreducible representation of $L^F_r$ satisfying the conditions of Theorem~\ref{thm} (resp.Theorem~\ref{Ellthm}). Then $R^{G}_{L,P,r}(\psi)$ is irreducible up to sign. 
\end{corollary}

Since \(X^{G}_{L,P,r}\) is stable under right multiplication by \(V_r\cap F(V_r)\), We define 
\[
Y^{G}_{L, P,r} := X^{G}_{L, P,r} / (V_r \cap F(V_r)) \cong 
\left\{ gV_r \in G_r / V_r \mid g^{-1}F(g) \in V_r \cdot F(V_r) \right\},
\]
the quotient map $$X^{G}_{L,P,r}\to Y^{G}_{L,P,r}$$ is the restriction
of a locally trivial fibration: $$G_r\longrightarrow G_r/(V_r\cap F(V_r)).$$ Hence it is also a locally trivial fibration. In particular, the smoothness of \(Y^{G}_{L,P,r}\) follows from the smoothness of \(X^{G}_{L,P,r}\); see, e.g., \cite[Lemma 2.2.3]{DMF}. It is not known in general whether $Y^{G}_{L,P,r}$ is affine. However, it is affine when $r=0$ (see \cite[\S 6.3]{Lu90} and \cite[Theorem 1.3]{PHL}).

\begin{corollary}
    Assume that $Y^{G}_{L,P,r}$ is an affine algebraic variety. Let $\psi$ be an irreducible representation of $L^F_r$  satisfying the conditions of Theorem~\ref{thm} (resp. Theorem~\ref{Ellthm}), and let $\rho$ be the unique irreducible constituent of $R^G_{L,B\cap L,r}(\psi)$. Then the following are equivalent:
    \begin{enumerate}
        \item $\rho$ appears in $H^i_c(X^G_{L,P,r}) \otimes_{L^F_r} \psi$;
        \item $i =  \dim(V_r) + \dim(V_r \cap F(V_r))$.
    \end{enumerate}
    Moreover, \(\rho\) appears in only a single degree, and with multiplicity one in that degree.
\end{corollary}

\begin{proof}
    There exist $i, j$ such that 
    \[
    \rho \subseteq H^i_c(X^{G}_{L, P,r}) \otimes_{L^F_r} \psi
    \quad \text{and} \quad 
    \overline{\rho} \subseteq H^j_c(X^{G}_{L, P,r}) \otimes_{L^F_r} \overline{\psi}.
    \]
    This implies the non-vanishing of the following intertwining space:
    \[
    \overline{\psi}^{\mathrm{op}} \otimes_{L^F_r} H^i_c(X^{G}_{L, P,r})^{\mathrm{op}} \otimes_{G^F_r} H^j_c(X^{G}_{L, P,r}) \otimes_{L^F_r} \psi \neq 0.
    \]
    So by equality \ref{key} in Theorem~\ref{thm} (resp. Theorem~\ref{Ellthm}), $i+j=2\dim(V_r) + 2\dim(V_r \cap F(V_r))$.
    Consider the Leray spectral sequence for the fibration $X^{G}_{L,P,r}\to Y^{G}_{L,P,r}$. Then, since we have assumed that $Y^{G}_{L,P,r}$ is affine, Artin's vanishing theorem forces that $$i\text{ } (\mathrm{resp.\text{ }} j)-2\dim(V_r \cap F(V_r))\geq \dim(V_r)-\dim(V_r \cap F(V_r)).$$
    Together with the previous equality this gives the result.
\end{proof}

\section{Decomposition of Elliptic Parahoric Deligne-Lusztig Representations}\label{4}
In this section we assume $T$ is elliptic and fix a character $\theta \colon T_r^F \to \overline{\mathbb{Q}}^\times_\ell$, Theorem~\ref{Howexist} allows us to choose a Howe factorization $\vec{\phi} = (\phi_{-1}, \phi_{0},\ldots, \phi_d)$ with depth parameters $(r_{-1}=0,r_{0},\dots,r_d=r)$.

By \cite[Lemma 3.1]{Nie}, there exists a $k$-rational Borel subgroup $B$ containing $T$ such that each Levi subgroup $G^i$ is standard with respect to $B$. 

We define the parabolic subgroup $P^i :=  G^i(B\cap G^{i+1})$ of $G^{i+1}$ with Levi subgroup $G^i$. This gives the following inclusion diagram:
\begin{equation*}
    \begin{tikzcd}[column sep=small, row sep=small]
        T = G^{-1} \ar[phantom]{r}{\subseteq} \ar[phantom]{d}[sloped]{\subsetneq} & 
        G^0 \ar[phantom]{r}{\subsetneq}\ar[phantom]{d}[sloped]{\subsetneq} & 
        G^1 \ar[phantom]{r}{\subsetneq}\ar[phantom]{d}[sloped]{\subsetneq} & 
        \cdots \ar[phantom]{r}{\subsetneq} & 
        G^{d-1} \ar[phantom]{r}{\subsetneq}\ar[phantom]{d}[sloped]{\subsetneq} & 
        G^d \ar[phantom]{d}[sloped]{=} \\
        B\cap G^{0} \ar[phantom]{r}{\subseteq} & 
        P^0 \ar[phantom]{r}{\subsetneq} & 
        P^1 \ar[phantom]{r}{\subsetneq} & 
        \cdots \ar[phantom]{r}{\subsetneq} & 
        P^{d-1} \ar[phantom]{r}{\subsetneq} & 
        G
    \end{tikzcd} 
\end{equation*}

For $0< i \leq d$, we define the map $\Psi_i:\CR((G^{i-1}_{r_{i-1}})^F)\to \CR((G^{i}_{r_i})^F)$ given by
$$\chi\mapsto {\phi_i}\otimes \mathrm{Inf}^{(G^{i}_{r_i})^F}_{(G^{i}_{r_{i-1}})^F} R^{G^i}_{G^{i-1},P^{i-1},r_{i-1}}(\chi)$$

In particular, we define $\Psi_0:\CR((G^{0}_{0})^F)\to \CR((G^{0}_{r_0})^F)$ as follows
\[
\chi\mapsto \phi_0\otimes\mathrm{Inf}^{(G^{0}_{r_0})^F}_{(G^{0}_{0})^F}\chi
\]

We denote the composition
\[
\Psi := \Psi_d \circ \cdots \circ \Psi_1 \circ \Psi_0.
\]

We write the decomposition of the classical Deligne-Lusztig representation as
$R^{G^0}_{T,B\cap G^0,0}(\phi_{-1}) = \sum_{\rho} m_{\rho}\rho$,
where $\rho$ runs over the irreducible constituents and $m_{\rho}$ denotes the multiplicity.

\begin{theorem}\label{decomp}
Suppose $T$ is elliptic over $k$ and $p$ satisfies condition $(\ast)$. Then 
    $$R^G_{T,B,r}(\theta)=\sum_{\rho} m_{\rho}\Psi(\rho).$$
    Moreover, the summands $\Psi(\rho)$ are pairwise non-isomorphic irreducible representations of $G^F_r$ up to sign.
    
\end{theorem}

\begin{proof}
    The equality follows directly from Proposition~\ref{Ell} and Lemma~\ref{Re}. To establish the irreducibility and distinctness, it suffices to prove the equation  $$\big\langle\Psi(\rho),\Psi(\rho')\big\rangle_{G^F_r}=\big\langle \rho,\rho'\big\rangle_{(G^{0}_{r_0})^F}$$
    for any irreducible constituents $\rho,\rho'$ of $R^{G^0}_{T,B\cap G^0,0}$.

    Define the intermediate characters: $$\theta_i={\phi_i}|_{T^F_{r_i}}\otimes{\phi_{i-1}}|_{T^F_{r_i}}\dots \otimes {\phi_{-1}}|_{T^F_{r_i}}$$
    These are $(G^i,G)$-generic characters of depth $r_i$
    
    The key observation is that $\Psi_i\circ \dots \circ \Psi_0(\rho)$ and $\Psi_i\circ \dots \circ \Psi_0(\rho')$ appear in $R^{G^i}_{T,B\cap G^i,r_i}(\theta_i)$. Therefore, they satisfy the hypotheses of Theorem~\ref{Ell}, which yields
     $$\big\langle\Psi_{i+1}(\chi_i),\Psi_{i+1}(\chi'_i)\big\rangle_{(G^{i+1}_{r_{i+1}})^F}=\big\langle \chi_i, \chi'_i\big\rangle_{(G^{i}_{r_{i}})^F}$$
    where $\chi_i = \Psi_i \circ \cdots \circ \Psi_0(\rho)$ and $\chi_i' = \Psi_i \circ \cdots \circ \Psi_0(\rho')$.
    By induction on $i$, we conclude that:$$\big\langle\Psi(\rho),\Psi(\rho')\big\rangle_{G^F_r}=\big\langle \rho,\rho'\big\rangle_{(G^{0}_{r_0})^F}$$
    This completes the proof of irreducibility and distinctness.
\end{proof}


\begin{thebibliography}{GKP00}

\bibitem{Comp}
C.~Bonnaf{\'e} and J.~Michel, \emph{Computational proof of the Mackey formula for $q > 2$}, Journal of Algebra, 327 (2011)
, 506--526

\bibitem{BoyarchenkoW_16}
M.~Boyarchenko and J.~Weinstein, \emph{Maximal varieties and the local Langlands correspondence for $\mathrm{GL}(n)$}, J. Amer. Math. Soc. 29 (2016), 177--236.

\bibitem{PHL}
B.~Patrick, J.~Hong, and D.~Lee. \emph{Geometry of regular semisimple Lusztig varieties} arXiv preprint arXiv:2504.15868 (2025).

\bibitem{Chan_siDL}
C.~Chan, \emph{The cohomology of semi-infinite Deligne-Lusztig varieties}, J. Reine Angew. Math. 768 (2020), 93--147.

\bibitem{Ch24}
\bysame, \emph{The scalar product formula for parahoric Deligne-Lusztig inductions}, arXiv:2405.00671.

\bibitem{CI21}
C.~Chan and A.~Ivanov, \emph{Cohomological representations of parahoric subgroups}, Represent. Theory 25 (2021), 1--26.

\bibitem{CI_loopGLn}
C.~Chan and A.~Ivanov, \emph{On loop Deligne-Lusztig varieties of Coxeter-type for inner forms of $\mathrm{GL}_n$}, Camb. J. Math. 11 (2023), 441--505.

\bibitem{CO}
C.~Chan and M.~Oi, \emph{Geometric L-packets of Howe-unramified toral supercuspidal representations}, J. Eur. Math. Soc. (2023), 1--62.


\bibitem{ChanOi_25b}
\bysame, \emph{Green functions for positive-depth Deligne-Lusztig induction}, arXiv:2506.04449.

\bibitem{Chen25}
Z.~Chen, \emph{On a stability of higher level Coxeter unipotent representations}, manuscripta math. 176, 5 (2025).

\bibitem{ChenS_17}
Z.~Chen and A.~Stasinski, \emph{The algebraisation of higher Deligne-Lusztig representations}, Selecta Math. (N.S.) 23 (2017), 2907--2926.

\bibitem{ChenS_23}
\bysame, \emph{The algebraisation of higher level Deligne-Lusztig representations II: odd levels}, arXiv: 2311.05354.

\bibitem{Crisp}
T.~Crisp, E.~Meir,  U.~Onn. \emph{A Variant of Harish-Chandra functors}. Journal of the Institute of Mathematics of Jussieu. 2019;18(5):993-1049. doi:10.1017/S1474748017000305

\bibitem{DL}
P.~Deligne and G.~Lusztig, \emph{Representations of reductive groups over finite fields}, Ann. Math. 103 (1976), 103--161.

\bibitem{DMF}
F.~Digne, J.~Michel, R.~Rouquier,, \emph{Cohomologie des variétés de Deligne–Lusztig}, Adv. Math. 209, 2 (2007), 749-822.

\bibitem{Howe}
R.~Howe, \emph{Tamely ramified supercuspidal representations of $\mathrm{GL}_n$}, Pacific J. Math. 73 (1977), 437--460.


\bibitem{Gerardin}
P.~G{\'e}rardin, \emph{Construction de s{\'e}ries discretes p-adiques},
Vol. 462. Springer-Verlag, Berlin-New York, 1975.

\bibitem{IvanovNie_24}
A.~Ivanov, S.~Nie, \emph{The cohomology of $p$-adic Deligne-Lusztig shcemes of Coxeter type}, J. Math. Jussieu. 12 (2025), no.4, 1429-1462.

\bibitem{IvanovNie_25}
\bysame, \emph{Convex elements and deep level Deligne-Lusztig varieties}, arXiv:2503.13412.

\bibitem{Kal}
T.~Kaletha, \emph{Regular supercuspidal representations}, J. Amer. Math. Soc. 32 (2019), 1071--1170.

\bibitem{Lu76}
G.~Lusztig, \emph{On the finiteness of the number of unipotent classes}, Inventiones mathematicae 34 (1976), 201--213.

\bibitem{Lu79}
\bysame, \emph{Some remarks on the supercuspidal representations of p-adic semisimple groups}, Automorphic forms, representations and L-functions, Proc. Symp. Pure Math 33 (1979), 171--175.

\bibitem{Lu84}
\bysame, \emph{Characters of reductive groups over a finite field},
Princeton University Press 107 (1984)

\bibitem{Lu90}
\bysame, \emph{Green functions and character sheaves}, Annals of Mathematics 131.2 (1990): 355-408.

\bibitem{Lu04}
\bysame, \emph{Representations of reductive groups over finite rings}, Represent. Theory 8 (2004), 1--14.

\bibitem{Nie}
S.~Nie, \emph{Decomposition of higher Deligne-Lusztig representations}, arXiv:2406.06430v3 (2024), 579--622.

\bibitem{Sta09}
A.~Stasinski, \emph{Unramified representations of reductive groups over finite rings}, Represent. Theory 13 (2009), 636--656.

\bibitem{Taylor}
J.~Taylor, \emph{On the Mackey formula for connected centre groups}, Journal of Group Theory. 21 (2018), 439--448

\bibitem{Yu}
J.-K.~Yu, \emph{Construction of tame supercuspidal representations}, J. Amer. Math. Soc. 14 (2001), 579--622.


\end{thebibliography}
\end{document}